\documentclass[a4paper,11pt]{article}

\usepackage{amsmath}
\usepackage{amssymb}
\usepackage{amsthm}
\usepackage{amsfonts}
\usepackage{mathdots}
\usepackage{enumitem}
\usepackage{comment}
\usepackage{enumitem}
\usepackage{hyperref}

\title{Lubell mass and induced partially ordered sets}
\author{Ar\`es M\'eroueh\footnote{Department of Pure Mathematics and Mathematical Statistics, Centre for Mathematical Sciences, Wilberforce Road, Cambridge, CB3 0WB, UK. E-mail: \href{mailto:a.j.meroueh@dpmms.cam.ac.uk}{a.j.meroueh@dpmms.cam.ac.uk}.}  }

\newtheorem{theorem}{Theorem}[section]
\newtheorem{conjecture}[theorem]{Conjecture}
\newtheorem{lemma}[theorem]{Lemma}

\newtheorem{claim}[theorem]{Claim}
\newtheorem{observation}[theorem]{Observation}

\newcommand{\N}{\mathbb{N}}
\newcommand{\Z}{\mathbb{Z}}

\newcommand{\C}{\mathcal{C}}
\newcommand{\U}{\mathcal{U}}
\newcommand{\D}{\mathcal{D}}
\newcommand{\pr}{\mathrm{P}}
\newcommand{\e}{\mathrm{E}}
\newcommand{\p}{\mathcal{P}}
\newcommand{\f}{\mathcal{F}}
\newcommand{\s}{\mathcal{S}}
\newcommand{\T}{\mathcal{T}}

\begin{document}
\maketitle
\begin{abstract}
We prove that for every partially ordered set $P$, there exists $c(P)$ such that every family $\f$ of subsets of $[n]$ ordered by inclusion and which contains no induced copy of $P$ satisfies $\sum_{F\in \f}1/{n\choose |F|}\leq c(P)$. This confirms a conjecture of Lu and Milans \cite{lumilans}. 
 
\end{abstract}
\section{Introduction}
A \emph{partially ordered set}, or \emph{poset}, is a set $P$ equipped with some partial order relation $\leq$. A typical example of a partially ordered set is the hypercube on $n$ vertices, namely $\p[n]$, the set of subsets of $\{1,2,\ldots, n\}$, equipped with the order relation $\leq$ so that for any two $A,B\in \p[n]$, $A \leq B$ if and only if $A\subseteq B$. A \emph{chain} in a poset $P$ is a collection $x_1, x_2,\ldots, x_k$ of elements of $P$ such that $x_1< x_2< \ldots < x_k$. The \emph{height} of $P$ is the maximal size of a chain in $P$.

In this paper we consider a Tur\'an-type question for posets: given a fixed poset $P$, what is the maximal size of a subset $\f$ of the hypercube which does not contain $P$ as a subposet? Let us clarify what is meant by containment in this context. Given posets $P$ and $P'$, we say that $P$ is \emph{weakly contained} in $P'$ if there exists an injective map $\psi: P\longrightarrow P'$ such that for any $x,y\in P$, $\psi(x) \leq_{P'} \psi(y)$ if $x\leq_P y$. We say that $P'$ \emph{strongly contains} $P$ if in fact for any $x,y\in P$, $x\leq_P y$ if and only if $\psi(x)\leq_{P'} \psi(y)$. We also say in this case that $P'$ contains $P$ as an \emph{induced} poset. Given a poset $P$, $ex(n,P)$ is defined by 
$$ex(n,P) = \max \{|\f|: \f\subseteq \p[n] \text { and }\f \text{ does not weakly contain } P \},$$
and $ex^*(n,P)$ is defined by 
$$ex^*(n,P) = \max \{|\f|: \f\subseteq \p[n] \text { and }\f \text{ does not strongly contain } P \}. $$
There has been a considerable amount of research devoted to determining the asymptotic behaviour of $ex(n,P)$ and $ex^*(n,P)$. The first result of this kind is Sperner's Theorem \cite{sperner}, which states that the maximal size of an antichain of $\p[n]$ is ${n\choose \lfloor n/2\rfloor}$. In other words, if $P_k$ is defined to be the chain poset of length $k$, then Sperner's result says that $ex(n,P_2) = {n\choose \lfloor n/2 \rfloor}$. This was later extended by Erd\H{o}s \cite{erdos}, who proved that $ex(n,P_k)$ is the sum of the $k-1$ largest binomial coefficients of order $n$. Of course, in the case of chains, $ex(n,P_k)$ and $ex^*(n,P_k)$ are the same number. A systematic study of $ex(n,P)$ for various specific posets was undertaken by a number of authors. For example, in the case of the diamond poset $D_2$ (which is also the hypercube of dimension 2), the best bound to date is $ex(n,D_2) = O\left((2.25 + o(1)){n\choose \lfloor n/2\rfloor}\right)$ (due to Kramer, Martin and Young \cite{kramaryoung}), but it is conjectured that $ex(n,D_2)=O\left((2+o(1)){n\choose \lfloor n/2\rfloor}\right)$.  It is not known whether the limit $\lim_{n\to\infty} ex(n,P)/{n \choose \lfloor n/2 \rfloor}$ or the limit $\lim_{n\to\infty} ex^*(n,P)/{n \choose \lfloor n/2 \rfloor}$ exists for every $P$. If these numbers do exists, then we denote them by $\pi(P)$ and $\pi^*(P)$, respectively. A central conjecture in this field is the following.
\begin{conjecture}
\label{laconjecture}
For each poset $P$, $\pi(P)$ exists and is equal to $e(P)$, where $e(P)$ is the maximal number $m$ such that for all $n$, the $m$ middle layers of $\p[n]$ do not weakly contain $P$.   
\end{conjecture}
This conjecture is attributed by Griggs, Li and Lu \cite{griggslilu} to Saks and Winkler, who made the (unpublished) observation that whenever $\pi(P)$ was known to exist it was also known to equal $e(P)$. The conjecture was proved to be true for posets whose Hasse diagram is a tree by Bukh \cite{bukh}.

In general, much less is known about induced containment than weak containment. Boehnlein and Jiang \cite{boehnleinjiang} extended Bukh's result \cite{bukh} to the induced case, i.e. they showed that the induced equivalent of Conjecture \ref{laconjecture} is also true for every tree poset. Carroll and Katona \cite{carrollkatona} proved that 
$${n\choose \lfloor n/2\rfloor}(1+1/n + \Omega(1/n^2))\leq ex^*(n,V_2) \leq {n\choose \lfloor n/2\rfloor} (1+2/n+O(1/n^2)), $$
where $V_2$ denotes the poset on three elements $a,b,c$ where $b$ and $c$ are incomparable and both larger than $a$. 

It follows at once from the result of Erd\H{o}s mentioned above that $ex(n,P)\leq (|P|+1){n \choose \lfloor n/2\rfloor}$. However it was unknown until recently  whether for every fixed poset $P$ there exists a constant $c(P)$ such that if $\f\subseteq \p[n]$ does not contain $P$ as an induced subposet, then $|\f| \leq c(P){n\choose \lfloor n/2 \rfloor}$. The existence of such a constant was conjectured to be true by Katona and by Lu and Milans \cite{lumilans}, and was proved by Methuku and P\'alv\"{o}lgyi \cite{methukupav} by means of a generalization of the Marcus-Tardos theorem about forbidden permutation matrices in 0-1 matrices \cite{marcustardos}. 
\begin{theorem}[Methuku, P\'alv\"{o}lgyi \cite{methukupav}] 
\label{indconst}
For every poset $P$ there exists $c(P)$ such that if $\f\subseteq \p[n]$ for some $n\in \N$ and $\f$ does not contain $P$ as an induced subposet, then $|\f|\leq c(P){n\choose \lfloor n/2\rfloor}$.
\end{theorem}

Given $\f\subseteq \p[n]$, the \emph{Lubell mass} of $\f$, denoted by $l(\f)$, is the quantity 
$$l(\f) = \sum_{F\in \f}1/{n\choose |F|}.$$
Notice that the Lubell mass of $\f$ is the expected number of times a maximal chain in $\p[n]$ chosen uniformly at random meets $\f$. Based on this observation, Lubell \cite{lubell} proved that if $\f\subseteq \p[n]$ is an antichain then $l(\f) \leq 1$, and this naturally implies Sperner's Theorem. More generally, Lubell's argument shows that if $\f$ contains no $P_k$ then $l(\f)\leq k-1$. The Lubell mass is a powerful tool for studying both $ex(n,P)$ and $ex^*(n,P)$, and has been used in various papers; for example Griggs, Li and Lu \cite{griggslilu} explicitly make use of it in order to bound $ex(n,D_2)$. Let us remark here that $l(\f)$ is technically a function of both $\f$ and the ground set $[n]$. Thus, to be precise, we should write $l_{[n]}(\f)$ instead of $l(\f)$. However, unless otherwise stated, the ground set is assumed to be $[n]$  and so we shall often simply write $l(\f)$. 

We see that, since $l(\f)\leq k-1$ for every family $\f$ not containing $P_k$, then $l(\f)\leq |P|-1$ for every family not containing a fixed poset $P$ weakly. The purpose of this paper is to prove a corresponding resut for induced containment.  
\begin{theorem}
\label{maintheorem}
For every poset $P$ there exists $c(P)$ such that if $\f\subseteq \p[n]$ for some $n\in \N$ and $\f$ does not contain $P$ as an induced subposet, then $l(\f)\leq c(P)$. 
\end{theorem}
Theorem \ref{maintheorem}, which strengthens Theorem \ref{indconst}, was conjectured to be true by Lu and Milans \cite{lumilans}. They proved it in a number of special cases including posets of height at most 2, ``series-parallel'' posets, and hypercubes of dimension at most 3. We refer the reader to their paper \cite{lumilans} for more details. 

It is easily seen that Theorem \ref{maintheorem} implies Theorem \ref{indconst}. In fact, the Lubell mass of a family $\f$ can be viewed as a weighting of the elements of the hypercube, so that each layer of the hypercube is given the same total weight. Thus one of the main advantages of Theorem \ref{maintheorem} is that it gives information about the low (and high) layers of $\f$, whereas Theorem \ref{indconst} does not. 

The rest of this paper is organized as follows. In Section \ref{universal}, we reduce the problem of finding an induced poset $P$ to finding an element of certain families of ``universal'' posets. In Section \ref{pivots}, we generalize the notion of \emph{pivots} which were introduced by Lu and Milans in \cite{lumilans}. Then in Section \ref{simplified}, we explain the basic idea behind the proof of Theorem \ref{maintheorem}, and this serves to demonstrate the usefulness of generalizing the concept of pivots. Section \ref{technicalsection} is devoted  to the proof of a key technical fact, Lemma \ref{nodense}, which will subsequently allow us to prove Theorem \ref{maintheorem} in Section \ref{mainproof}. 
 
\section{Universal posets}
\label{universal}
Throughout this paper we shall use the following standard notation. Given $n,m \in \N$, $[n]$ denotes the set $\{1,2,\ldots, n\}$, and if $X$ is a set then $X^{(m)}$ denotes the set $\{Y\subseteq X: |Y| = m\}$; the elements of $X^{(m)}$ will be referred to as the $m$-subsets of $X$. Also, $X^{(\leq m)}$ denotes $\bigcup_{i=0}^m X^{(i)}$. Given $a,b\in \Z$ with $a\leq b$, $[a,b]$ denotes the set $\{a,a+1,\ldots, b\}$.

Let $n\in \N$. Let $Q_n^+$ be the partially ordered set obtained by ordering $\p[n]$ so that $x\leq y$ if and only if $x\subseteq y$. Thus $Q_n^+$ is simply the hypercube with the usual partial order. Let $Q_n^-$ be the partially ordered set obtained by ordering $\p[n]$ so that $x\leq y$ if and only if $x\supseteq y$. Obviously $Q_n^+$ and $Q_n^-$ are isomorphic, but it will be convenient in our proofs to have access to both of them. 

We begin with a simple lemma about partially ordered sets, observed in \cite{lumilans}.  

\begin{lemma}
\label{universality}
Let $P$ be a partially ordered set. Then $P$ is an induced subposet of $Q_{|P|}^+$.
\end{lemma}

\begin{proof}

Identify the elements of $[|P|]$ with those of $P$, i.e. let $n_x:x\in P$ be an enumeration of $[|P|]$. Consider the map $\psi:P\longrightarrow [|P|]$ defined by $\psi(x) = \{n_z: z\in P,\,z\leq x \}$. Let us check that $\psi$ is an injective map which respects the order relation on $P$. To do so, it suffices to prove that if $x,y\in P$ and $x<y$ then $\psi(x)\subsetneq \psi(y)$, and if $x,y$ are incomparable then $\psi(x)$ and $\psi(y)$ are incomparable too. So let $x,y\in P$. If $x<y$ then $\psi(x) = \{n_z:z\in P,\, z\leq x \}\subseteq \{ n_z:z\in P,\, z\leq y\} = \psi(y)$. Besides, $n_y\not \in \psi(x)$ as $y\not\leq x$. Thus $\psi(x)\subsetneq \psi(y)$. If $x,y$ are incomparable then since $x\not\leq y$, $n_x\not\in \psi(y)$ and since $y\not\leq x$, $n_y\not\in \psi(x)$. Hence $\psi(x)$ and $\psi(y)$ are incomparable. 
\end{proof}

By Lemma \ref{universality}, in order to find an induced copy of a poset $P$ it is enough to find an induced copy of $Q_{|P|}^+$. However, an important idea of the proof of Theorem \ref{maintheorem}, rather than trying to find a copy of $Q_{|P|}^+ $ directly,  is to find a copy of a ``dense'' subposet which itself contains the desired copy of $Q_{|P|}^+$. We formalize this idea in what follows. 

Let $m,n\in \N$. We let $U(n,m)$ denote the partially ordered set obtained by ordering the elements of $[n]^{(\leq m)}$ so that $x\leq y$ if and only if $x\subseteq y$. Similarly $D(n,m)$ denotes the partially ordered set obtained by ordering the elements of $[n]^{(\leq m)}$ so that $x\leq y$ if and only if $x\supseteq y$. 

Let $\epsilon >0$. We let $\U(n,m,\epsilon)$ be the set of induced subposets $U$ of $U(n,m)$ such that $|U\cap [n]^{(i)}|\geq (1-\epsilon){n \choose i}$ for every $i$, $0\leq i\leq m$.  $\D(n,m,\epsilon)$ is the set of induced subposets $D$ of $D(n,m)$ such that $|D\cap [n]^{(i)}|\geq (1-\epsilon){n\choose i}$ for every $i$, $0\leq i \leq m$.  

\begin{lemma}
\label{universality2}
Let $m\in \N$. There exists $\epsilon >0$ such that for all $n\geq 2m$, every $U\in \U(n,m,\epsilon)$ contains an induced copy of $Q_m^+$, and every $D\in \D(n,m,\epsilon)$ contains an induced copy of $Q_m^-$. 
\end{lemma}

\begin{proof}
Let $\epsilon = 1/(2m)^{m+1}$. Let $n\geq 2m$.  Let $p = 2m/n$. Let $U\in \U(n,m,\epsilon)$. Let  $\overline{U} = \{x\in [n]^{(\leq m)}:x\not\in U\}$. To prove the lemma, it is clearly sufficient to find a subset of $[n]$ of size $m$ not containing any element of $\overline{U}$. Choose a random subset $X$ of $[n]$, each element of $[n]$ being selected independently of the others with probability $p$. Let $Z$ be the number of elements of $[n]$ selected, so $\e(Z) = 2m$, and for $0\leq i\leq m$ let $Z_i$ be the number of elements of $\overline{U}\cap [n]^{(i)}$ contained in $X$. By assumption, $|\overline{U}\cap [n]^{(i)}|\leq \epsilon {n\choose i}$, so $\e(Z_i)\leq \epsilon p^i {n\choose i}\leq 1/(2m)$ by our choice of $\epsilon$ and $p$. Note that in fact $Z_0 = 0$ since $\emptyset \in U$ (because it is the only element of size 0 and $\epsilon >0$). Therefore $\e(Z - \sum_{0\leq i\leq m}Z_i)\geq m$ and so there exists a set $X$ for which $Z- \sum_{0\leq i\leq m}Z_i\geq m$. After removing from $X$ one vertex from every element of $\overline{U}$ contained in $X$ we are left with a set $X'$ of size at least $m$ and which does not contain any element of $\overline{U}$. The statement about $D(n,m,\epsilon)$ follows by symmetry.
\end{proof}

\section{Generalized pivots}
 \label{pivots}
Let $n\in \N$ and let $\f$ be a family of subsets of $[n]$. Let $r\in \N$. An \emph{$r$-pivot} of $A$ is an element $X$  of $A^{(r)}$ such that there exists an element $Y$ of $([n] \backslash A)^{(r)}$ and an element $B$ of $\f$ such that $B = (A \backslash X) \cup Y$. In other words, we can obtain a subset in $\f$ by deleting the $r$ elements of $X$ from $A$ and replacing them by $r$ new elements. We say that $B$ is a \emph{witness} of $X$ (being a pivot of $A$). If $r=0$ then we view $\emptyset$ as being a 0-pivot for any $A\in \f$, $A$ itself being the witness that $\emptyset$ is a pivot of $A$. An \emph{$r$-anti-pivot} is an element $Y$  of $([n] \backslash A)^{(r)}$ such that there exists an element $X$ of $A^{(r)}$ and an element $B$ of $\f$ such that $B = (A \backslash X) \cup Y$. The concept of $r$-pivots generalizes that of pivots introduced in \cite{lumilans} by Lu and Milans. A simple but quite important observation about $r$-pivots is as follows.

\begin{observation}
\label{obs}
Let $A\subseteq [n]$ and let $X\in A^{(r)}$ be an $r$-pivot of $A$ with witness $B$. Then for any $F\subseteq A$, $F\subseteq B$ if and only if $F\cap X = \emptyset$. Likewise, if $Y\in ([n] \backslash A)^{(r)}$ is an $r$-anti-pivot of $A$ with witness $B$, then for any $F\supseteq A$, $B\subseteq F$ if and only if $Y\subseteq F$.  
\end{observation} 

We generalize another related concept from \cite{lumilans}. Let $\gamma\in (0,1]$ be a real number. We say that a set $A\in \f$ is \emph{$(\gamma,r)$-flexible} (in $\f$) if it has at least $(1-\gamma){|A|\choose r}$ $r$-pivots, and is \emph{$(\gamma,r)$-anti-flexible} if it has at least $(1-\gamma){n-|A|\choose r}$ $r$-anti-pivots. 

An important ingredient of the proof of Theorem \ref{maintheorem} in the case of posets of height two in \cite{lumilans} was that a family $\f$ containing no large set as well as no $(\gamma,1)$-flexible sets has bounded Lubell mass. We prove that the same is true for generalized pivots. 

\begin{lemma}
\label{pivotsbound}
Let $r,n\in \N$, let $\gamma \in (0,1]$. Let $\f \subseteq \p[n]$ be a family not containing any $(\gamma, r)$-flexible set and not containing any set of size more than $n/2$. Then $l(\f)\leq f(\gamma,r)$, where $f(\gamma,r) = r + 2r^2\gamma^{-1}$. 
\end{lemma} 
\begin{proof}
Let $r\leq k\leq n/2$. Construct a bipartite graph with vertex classes $\f \cap[n]^{(k)}$ and $[n]^{(k-r)}$ by drawing an edge between $A\in \f\cap [n]^{(k)}$ and $C\in [n]^{(k-r)}$ if $C\subseteq A$ and $A\backslash C$ is not an $r$-pivot of $A$. Let $G$ be the subgraph obtained by deleting the vertices of degree zero in $[n]^{(k-r)}$, and let $U$ and $V$ be the two vertex classes of $G$, so that $U = V(G)\cap [n]^{(k)}$ and $V = V(G) \cap [n]^{(k-r)}$. Let us count $e(U,V)$, the number of edges of this bipartite graph, in two different ways. First, for every $A\in U$, since $A$ is not $(\gamma, r)$-flexible, $A$ must be joined to at least $\gamma{k\choose r}$ elements of $V$. Thus, $e(U,V) \geq \gamma{k\choose r}|U|$. Now let $C\in V$, so there exists $A\in U$ such that $CA$ is an edge of $G$. Then any element of $U$ other than $A$ containing $C$ must intersect $A\backslash C$ (else $A\backslash C$ would be a pivot of $A$). Now for each $x\in A\backslash C$, there are at most ${n-k+r-1\choose r-1}$ such sets containing $x$, so $C$ has degree at most $1 + r{n-k+r-1\choose r-1}\leq 2r{n-k+r-1\choose r-1} = 2r^2{n-k+r\choose r}/(n-k+r)$ in $G$. Hence $e(U,V)\leq 2r^2{n \choose k-r}{n-k+r \choose r}/(n-k+r) = 2r^2{n\choose k}{k\choose r}/(n-k+r)$. Thus by combining the two bounds on $e(U,V)$, we have $|U|/{n\choose k}\leq 2r^2/(\gamma(n-k+r))$. Therefore 
\begin{align*}
l(\f) &= \sum_{k=0}^{n}\frac{|\f\cap [n]^{(k)}|}{{n\choose k}}\\
&\leq  r + \sum_{k=r}^{n/2}\frac{|\f\cap [n]^{(k)}|}{{n\choose k}}\\
&\leq  r + 2r^2\gamma^{-1}\sum_{k=r}^{n/2}\frac{1}{n-k}\\
&\leq  r + 2r^2\gamma^{-1}.\qedhere
\end{align*}
\end{proof}

\section{A first attempt}
\label{simplified}
In order to motivate the introduction of $r$-pivots, as well as to explain the basic idea behind the proof of Theorem \ref{maintheorem}, we shall give in this section a brief sketch of a proof of Theorem \ref{maintheorem} under two simplifying (but false!) assumptions. To this end, we first introduce some notation as well as a basic and important lemma found in Lu and Milans \cite{lumilans}. We include a proof of it in an effort to make this paper self-contained. 

Let $n\in \N$ and let $\f\subseteq \p[n]$. Let $A,B\subseteq [n]$ with $B\subseteq A$. We let $l_{B,A}(\f)$ denote the expected number of times a random full chain in the interval $[B,A]$ meets $\f$ (the interval $[B,A]$  of $\p[n]$ is the set $\{X\in \p[n]: B\subseteq X\subseteq A \}$). Also, we let $\f_{B,A}$ denote the family $\{F\backslash B:F\in \f, B\subseteq F \subseteq A\}$. Notice that $l_{B,A}(\f) = l_{\emptyset,A\backslash B}(\f_{B,A})$. In other words, $l_{B,A}(\f)$ is the Lubell mass of $\f_{B,A}$ viewed as a family of subsets of $A\backslash B$. 

\begin{lemma}[Lu, Milans \cite{lumilans}]
\label{centred}
There exist $A, B \in \f$ such that $l_{\emptyset,A}(\f)\geq l(\f)$ and $l_{B,[n]}(\f)\geq l(\f)$. 
\end{lemma}

\begin{proof}
Choose a maximal chain $\mathcal{C}$ in $\p [n]$ uniformly at random. Let $Z$ be the number of times $\C$ meets $\f$, so that $\e(Z) = l(\f)$. Given $A\in \f$, let $E_A$ be the event that $A$ is the largest element of $\f$ that $\C$ contains. Also let $E_*$ be the event that $\C$ meets no elements of $\f$. We then have $\e(Z) = \pr(E^*)\e(Z|E^*)+\sum_{A\in \f}\pr(E_A)\e(Z|E_A)$. Now obviously $\e(Z|E^*) = 0$, and so $\e(Z) =\sum_{A\in \f}\pr(E_A)\e(Z|E_A)$.  But, clearly, $\e(Z|E_A) =l_{\emptyset, A}(\f)$ for every $A\in \f$. Hence there must exist $A$ such that $l_{\emptyset, A}(\f)\geq l(\f)$. The second part of the statement follows by symmetry.  
\end{proof}

Now that we have stated Lemma \ref{centred}, we are ready to give our incorrect proof of Theorem \ref{maintheorem}. Let $n\in \N$ and let $\f\subseteq \p[n]$. Let $\f_+ = \{ F\in \f:|F|\geq n/2$ and let $\f_- = \{ F\in \f:|F|\leq n/2\}$. Our first assumption shall be that whatever $\f$ is, it is always true that $l(\f_-)\geq l(\f)/2$. Our second assumption shall be that Lemma \ref{pivotsbound} is also true if $\gamma = 0$, i.e. there exists some constant $b_r$ depending on $r$ only such that $l(\f) \leq b_r$ for every family $\f$ containing no set $A$ for which each element of $A^{(r)}$ is a pivot, as well as no set of size more than $n/2$. 

Suppose then that $P$ is some arbitrary poset. Let $m = |P|$. Suppose that $l(\f)
\geq 2b_1 + 4b_2 + \cdots + 2^{m}b_m + 2^{m+1}(m+1)$. We begin by finding sequences $A_0\supseteq A_1\supseteq A_2\supseteq \cdots \supseteq A_m$ and $\f_0, \f_1,\f_2,\ldots, \f_m$ 
such that for each $i\in [0,m]$, $\f_i\subseteq \f$, $F\subseteq A_i$ for 
each $F\in \f_i$, and finally $A_i$ is a $(0,i)$-flexible set in $\f_{i-1}$.  
Moreover, we make sure that $l_{\emptyset,A_m}(\f_m)\geq m+1$. To do so, note by Lemma \ref{centred} that there 
exists $A_0\in \f$ such that $l_{\emptyset,A_0}(\f)\geq l(\f)$. Let $\f_0 = \f_
{\emptyset,A_0}$, so that $l(\f_0) \geq l(\f)$ when viewed as a family of subsets 
of $A_0$. By our first assumption, $l_{\emptyset,A_0}((\f_0)_-)\geq l(\f_0)/2 = 
b_1 + 2b_2 + \cdots + 2^{m-1}b_m+ 2^{m}(m+1)$. Now let $(\f_0)_-^*$ be the set of 
elements of $\f$ which are not $(0,1)$-flexible in $(\f_0)_-$. Notice that a pivot 
of $A$ in $(\f_0)_-^*$ is certainly also a pivot of $A$ in $(\f_0)_-$. Therefore 
the elements of $(\f_0)_-^*$ are not $(0,1)$-flexible in $(\f_0)_-^*$ either. 
Hence by our second assumption, $l_{\emptyset,A_0}((\f_0)_-^*)\leq b_1$. Therefore 
by considering  $(\f_0)_- \backslash (\f_0)_-^* $ we see that by Lemma \ref
{centred}, there exists $A_1\in (\f_0)_-$ which is $(0,1)$-flexible in $(\f_0)_-$ 
and moreover $l_{\emptyset, A_1}(\f_0)\geq l_{\emptyset, A_1}((\f_0)_-) - b_1 \geq l(\f_0)/2 - b_1 \geq b_2 + 2b_3 
+ \cdots + 2^{m-2}b_m+ 2^{m-1}(m+1)$. We let $\f_1 = (\f_0)_{\emptyset,A_1}$, and iterate 
this procedure, now finding a set $A_2\in \f_1$ which is $(0,2)$-flexible and such 
that $l_{\emptyset,A_2}(\f_1)\geq b_3 + 2b_4 + \cdots + 2^{m-3}
b_m+ 2^{m-2}(m+1)$, and so on. Iterating this procedure $m+1$ times yields the desired sequences. 
 
 Let $X = A_m$. Since $l_{\emptyset,A_m}(\f_m)\geq m+1$, $X$ has size at least  $m$. For each $i$, $0\leq i \leq m$, and each $x\in 
X^{(i)}$, since $x$ is an $i$-pivot of $A_i$ in $\f_{i-1}$, there exists a witness 
$w_x\in \f_{i-1}$ for the fact that $x$ is an $i$-pivot of $A_i$. Let $W = \bigcup_
{x\in X^{(\leq m)}}\{w_x\}$. Notice that $W$ is a subset of $\f$. Let us order $X^
{(\leq m)}$ by reverse inclusion, so that that $X$, as a partially ordered set, is 
isomorphic to $D(|X|,m)$. Consider the injective map $\psi: X^{(\leq m)}
\longrightarrow W$ which sends $x\in X^{(\leq m)}$ to its witness $w_x$. We claim 
that $\psi$ preserves the order relation on $X^{(\leq m)}$. Indeed, suppose that 
$x,y\in X^{(\leq m)}$ where $w_x\in \f_{i-1}$ for some $i$ and $w_y \in \f_{j-1}$ 
for some $j$ and without loss of generality $i\leq j$. If $i=j$ and $x\neq y$, 
then $x$ and $y$ are incomparable in $X^{(\leq m)}$.  It is also easy to see that 
$w_x$ and $w_y$ are two distinct elements of $\f$ of the same size and hence are 
incomparable in $\f$ too. So let us assume that $i<j$ (and so $x\neq y$). Clearly since 
$|x| < |y|$, either $x>y$ (if $x\subsetneq y$) or $x$ and $y$ are incomparable. If 
$x\subseteq y$, then by Observation \ref{obs}, $w_y \subsetneq w_x$ since $w_y
\subseteq A_i$. If $x$ and $y$ are incomparable in $X^{(\leq m)}$ then there 
exists $z\in x\backslash y$, but then since $x\subseteq X\subseteq A_j$, it must 
be the case that $z\in w_y$, and so since $z\not\in w_x$, $w_x$ and $w_y$ must be 
incomparable. This proves that $\psi$ preserves the order relation, and hence that 
$\f$ contains an induced copy of $D(|X|,m)$, hence one of $Q^-_m$ (as $|X|\geq m)
$, and so one of $P$ too. 

Of the two assumptions that we made, the first one is the one which is most easily dealt with. Indeed, it is clear that for any $\f\subseteq \p[n]$, either $l(\f_-)\geq l(\f)/2$ or $l(\f_+)\geq l(\f)/2$. Therefore, in the proof of Theorem \ref{maintheorem}, we shall resort to a ``two passes'' argument; there will be $2m+1$ steps (rather than $m+1$) and for at least $m+1$ steps we shall know that the same alternative held (say $l(\f_-)\geq l(\f)/2$), which, as it turns out, shall be enough to find a copy of $P$. 

The second assumption is more difficult to overcome. Indeed, it is not possible to force all $r$-subsets of some element of $\f$ to be $r$-pivots. However, it is possible, by Lemma \ref{pivotsbound}, to force a \emph{very large} proportion of them to be $r$-pivots. How can one then make use of this? The key idea is that, having ensured that a very large proportion of the $r$-subsets of a set $A$ are $r$-pivots, we can find a large subfamily of $\f_{\emptyset,A}$ such that each member $F$ of this family is such that a very large proportion of its $r$-subsets are pivots of $A$. This will allow us to carry out the iterative procedure described above. In the next section, we shall make this idea more precise.   

\section{Families containing no fat sets}
\label{technicalsection} 

Let us begin with an important definition. Let $r,n\in \N$. Let $\s\subseteq [n]^{(r)}$. Let $X\subseteq [n]$.  Let $\epsilon >0$. We say that $X$ is \emph{$(\epsilon,\s)$-fat} if $|X^{(r)}\cap\s|\geq (1-\epsilon) {|X| \choose r}$. 

Our goal in this section is to prove Lemma \ref{nodense}, which says that a family $\f$ containing no $(\epsilon,\s)$-fat set has bounded Lubell mass provided $\s$ is large enough. Essentially all the work required to prove Lemma \ref{nodense} is contained in Lemma \ref{technical}, where we use the language of probability to bound the proportion of elements of a fixed layer of the hypercube which contain too many elements of a small subset $\T$ of $[n]^{(r)}$.

Let $m,k,n$ be positive integers with $\max(m,k) \leq n$. The \emph{hypergeometric distribution} with parameters $m,k,n$ is the probability distribution of the random variable $Z = |X\cap [k]|$ where $X$ is  chosen uniformly at random in $[n]^{(m)}$. In other words, $Z$ counts the number of elements  which lie in $[k]$ of an $m$-subset of $[n]$ chosen uniformly at random. By Theorem 2.10 of \cite{janluru}, the following standard concentration inequality holds. 

\begin{lemma}
\label{hypgeom}
Let $m,k,n$ be positive integers with $\max(m,k) \leq n$, and let $Z$ be a hypergeometric random variable with parameters $m,k,n$. Let $t\geq 0$. Then
$$\pr\left(Z\geq \e(Z) + t\right) \leq \exp(-2t^2/m) .$$
\end{lemma}

Lemma \ref{hypgeom} enables us to prove the intuitively true fact if $\T\subseteq [n]^{(r)}$ is small then with very high probability a uniformly random $m$-subset of $[n]$ contains few elements of $\T$.  

\begin{lemma}
\label{technical}
Let $\epsilon\in (0,1]$. Let $r\in \N$. There exists $c(\epsilon,r)>0$, $\eta(\epsilon,r)>0$ and $m_0(\epsilon,r)$ such that, for each $n\geq m\geq m_0(\epsilon,r)$, if $\T\subseteq [n]^{(r)}$ satisfies $|\T|\leq \eta(\epsilon,r){n\choose r}$, then 
$$\pr\left(|X^{(r)}\cap \T|> \epsilon{m\choose r}\right)\leq \exp(-c(\epsilon,r)m), $$
where $X$ is chosen uniformly at random in $[n]^{(m)}$. 
\end{lemma}

\begin{proof}
It will be convenient here to view $\T$ as an $r$-uniform hypergraph with vertex set $[n]$. In the light of this, let us remind the reader of two standard definitions. Let $v\in [n]$. The \emph{link graph} of $v$ in $X$ is the set $L_X(v) =\{ e\backslash\{v\}:e \in \T,\, v\in e \text{ and }e\subseteq X \}$. The \emph{degree} of $v$ in $X$, denoted by $d_X(v)$, is defined to be $|L_X(v)|$. When $X = [n]$ we simply write $L(v)$ and $d(v)$. 

For $r=0$, we can set $\eta(\epsilon,0) = 1/2$, $c(\epsilon,0)=1$ and $m_0(\epsilon,0) = 0$ for any $\epsilon$ (in fact, any $\eta(\epsilon,0)<1$ and any $c(\epsilon,0)>0$ would do). So let us focus on the case $r\geq 1$. We shall prove the lemma by induction on $r$. First consider the base case $r=1$. Set $\eta(\epsilon,1) = \epsilon/2$ and $m_0(\epsilon,1)=1$ and suppose $\T\subseteq [n]$ has size at most $\eta(\epsilon,1)n$. Let $m\geq m_0(\epsilon,1)$, and choose $X\in [n]^{(m)}$ uniformly at random. Then $|X\cap \T|$ is hypergeometrically distributed with parameters $m$, $k=|\T|$, $n$, and $\e(|X\cap \T|) = km/n \leq \epsilon m/2$. Therefore, by Lemma \ref{hypgeom}, 
\begin{align*}
\pr(|X\cap \T|>\epsilon m) &=  \pr(|X\cap \T|>\epsilon m/2 + \epsilon m/2) \\
&\leq \pr(|X\cap \T|\geq \e(|X\cap \T|) + \epsilon m/2) \\
& \leq \exp(-\epsilon^2 m/2).
\end{align*}
Thus the base case holds if we set $c(\epsilon,1) = \epsilon^2/2$.

Suppose now that the induction hypothesis holds for all $r'$ with $1\leq r'\leq r-1$, and let us prove that it holds for $r$. In other words, we assume that values of $c$, $\eta$ and $m_0$ exist and satisfy the requirements of the lemma for any $\epsilon>0$ and $r'\in [1,r-1]$. Let $\delta_1 = \eta(\epsilon/2,r-1)$, let $\delta_2 = \eta(\epsilon/2,1)$ and set $\eta(\epsilon,r)=\delta_1\delta_2$.  Let $\T \subseteq [n]^{(r)}$ with $|\T|\leq \eta(\epsilon,r){n\choose r}$. Let $c(\epsilon,r) = \min\{c(\epsilon/2,1),c(\epsilon/2,r-1)\}/2$. Let $m_*$ be sufficiently large that 
$$\exp(-c(\epsilon/2,1)m) + m\exp(-c(\epsilon/2,r-1)(m-1)) \leq \exp(-c(\epsilon,r)m)$$ 
for all $m\geq m_*$. Let $m_0(\epsilon,r)= \max\{m_0(\epsilon/2,1),m_0(\epsilon/2,r-1)+1,m_*\}$. Choose $X\in [n]^{(m)}$ uniformly at random. 

Let us bound the probability that $|X^{(r)}\cap \T| > \epsilon{m\choose r}$. We define two sets $V_1$ and $V_2$ of elements of $[n]$ as follows. Let $V_1 = \{v\in [n]:d(v)> \delta_1{n-1\choose r-1}\}$ and let $V_2 = \{v\in [n]:d(v)\leq \delta_1{n-1\choose r-1} \}$. $V_1$ and $V_2$ partition $[n]$ and we view the elements of $V_1$ as having high degree and those of $V_2$ as having low degree. 

Let $A$ be the event that $|X \cap V_1|> \epsilon m/2$. We wish to bound the probability that $A$ occurs. For this, we first find an upper bound on $|V_1|$. Each vertex of $V_1$ contains at least $\delta_1{n-1\choose r-1}$ elements of $\T$, so $r|\T| \geq \delta_1{n-1\choose r-1} |V_1|$, hence 
\begin{align*}
|V_1|& \leq \frac{r|\T|}{\delta_1{n-1\choose r-1}}\\
&\leq \frac{r\eta(\epsilon,r){n\choose r}}{ \delta_1{n-1\choose r-1}}\\
&=\eta(\epsilon,r)n/\delta_1 \\
&=\delta_2n.
\end{align*}
By the choice of $\delta_2$ and since $m\geq m_0(\epsilon/2,1)$, we then have $\pr(A) \leq \exp(-c(\epsilon/2,1)m)$.

For $v \in V_2$, let $B_v$ be the event that $|X^{(r-1)}\cap L(v)|> \epsilon{m-1\choose r-1}/2$, conditional on $v\in X$. Fix $v\in V_2$. Let us bound the probability of $B_v$. Given that $v\in X$, $X\backslash \{ v \}$ is a uniformly random subset of $[n]\backslash \{v\}$ of size $m-1$. Since $|L(v)| \leq \delta_1{n-1\choose r-1}$ and $\delta_1 = \eta(\epsilon/2,r-1)$, and since $m-1\geq m_0(\epsilon/2,r-1)$, the probability that $B_v$ occurs is no more than $\exp(-c(\epsilon/2,r-1)(m-1))$. Now let $B$ be the event that there is some element $v$ of $V_2$ which belongs to $X$ and for which $d_X(v)> \epsilon{m-1\choose r-1}/2$. By a simple union bound,
\begin{align*} 
\pr(B) &\leq \sum_{v\in V_2}\pr(v\in X)\pr(B_v)\\
&\leq |V_2|(m/n)\exp(-c(\epsilon/2,r-1)(m-1))\\
& \leq m\exp(-c(\epsilon/2,r-1)(m-1)).
\end{align*}

If neither $A$ nor $B$ occurs, then 
\begin{align*}
r|X^{(r)}\cap \T| &= \sum_{v\in X\cap V_1}d_X(v) + \sum_{v\in X\cap V_2}d_X(v)\\
&\leq |X\cap V_1| {m-1\choose r-1} + m\max_{v\in X\cap V_2} d_X(v)\\
& \leq (\epsilon m/2) {m-1\choose r-1} + (\epsilon m/2){m-1\choose r-1} \\
&= \epsilon m{m-1\choose r-1}.
\end{align*}
Thus $|X^{(r)}\cap \T|\leq \epsilon {m\choose r}$. This implies that 
\begin{align*}
\pr\left(|X^{(r)}\cap \T| > \epsilon{m\choose r}\right) &\leq \pr(A) + \pr(B) \\
&\leq \exp(-c(\epsilon/2,1)m) + m\exp(-c(\epsilon/2,r-1)(m-1)) \\
&\leq \exp(-c(\epsilon,r)m),
\end{align*} 
the last inequality holding since $m\geq m_*$. 
\end{proof}

\begin{lemma}
\label{nodense}
Let $r\in \N$. Let $\epsilon>0$. There exists $h(\epsilon,r)$ such that if $\s\subseteq [n]^{(r)}$ and $|\s|\geq (1-\eta(\epsilon,h)){n\choose r}$ then $l(\f) \leq h(\epsilon,r)$ for every $\f\subseteq \p[n]$ which contains no $(\epsilon,\s)$-fat set, where $\eta(\epsilon,r)$ is the constant defined in Lemma \ref{technical} and $h$ is a constant depending on $\epsilon$ and $r$ only. 
\end{lemma} 
\begin{proof}
If $\f\subseteq \p[n]$ contains no $(\epsilon,\s)$-fat set then $|F^{(r)}\cap \overline{\s}|> \epsilon {m\choose r}$ for every $F\in \f$, where $\overline{\s}$ is the complement of $\s$ in $[n]^{(r)}$, i.e. $\overline{\s} = \{x\in [n]^{(r)}:x\not\in \s \}$. By assumption $|\overline{\s}| \leq \eta(\epsilon,r){n\choose r}$, hence by Lemma \ref{technical} applied to $\T = \overline{\s}$, we have 
\begin{align*}
l(\f) &= \sum_{m=0}^{n}\frac{|\f\cap [n]^{(m)}|}{{n\choose m}}\\
&\leq m_0(\epsilon,r) + \sum_{m=m_0(\epsilon,r)}^n \exp(-c(\epsilon,r)m)\\
&\leq m_0(\epsilon,r) + 1/(1-e^{-c(\epsilon,r)}).
\end{align*}
Thus the lemma certainly holds if we let $h(\epsilon,r) = m_0(\epsilon,r) + 1/(1-e^{-c(\epsilon,r)})$.
\end{proof}

\section{Proof of Theorem \ref{maintheorem}}
\label{mainproof}

We are now ready to prove Theorem \ref{maintheorem}. The backbone of the proof is the same as the argument described in Section \ref{simplified}. We stress this point because it is easy to get lost in the details of the proof below. Lemma \ref{intermediate} represents one step of the iteration, while Lemma \ref{sequences} is in essence the result of applying Lemma \ref{intermediate} (at most) $2m+1$ times. After proving these two lemmas, we finish the proof of the theorem by extracting the desired poset from the sequences that we constructed.  

Let $P$ be a partially ordered set. Let $m = |P|$. By Lemmas \ref{universality} and \ref{universality2} there exists $\epsilon >0$ such that for each $n\geq 2m$, every $U\in \U(n,m,\epsilon)$ and every $D\in \D(n,m,\epsilon)$ contains an induced copy of $P$. 

Let us define some constants. Here $f$ represents the contant defined in Lemma \ref{pivotsbound}, $\eta$ represents the constant defined in Lemma \ref{technical} and $h$ represents the constant defined in Lemma \ref{nodense}. Let 
$$\epsilon_1 = \epsilon$$ 
and for $2\leq j\leq 2m+1$ let 
$$\epsilon_{j}=\min\{ \epsilon_{j-1}, \,\eta(\epsilon_{j-1},i):i\in [0,m]\}.$$ 
For $2\leq j\leq 2m+1$, let 
$$q_j = \max_{i\in [0,m]}h(\epsilon_{j-1},i).$$
Let $$q = \max_{j\in [2,2m+1]}q_j.$$ 
Let $$p = \max_{i\in [0,m],\, j\in [1,2m+1]}f(\epsilon_j,i).$$
The constants $\epsilon_j$, $q$ and $p$ above are defined precisely so that the following statement is true.  
\begin{lemma}
\label{intermediate}
Let $0 \leq d \leq 2m$. Let $0\leq a,b\leq m$. Let $n\in \N$. If $d\geq 1$, let $\s_0\subseteq [n]^{(r_0)}$, $\s_1 \subseteq [n]^{(r_1)}$,\ldots, $\s_{d-1} \subseteq [n]^{(r_{d-1})}$ for some $0\leq r_0, r_1,\ldots, r_{d-1}\leq m$, and suppose that $\s_i$ is $(\epsilon_{2m+2-d},[n]^{(r_i)})$-fat for all $i$, $0\leq i\leq d-1$.  Let $\f\subseteq \p[n]$ with $l(\f)> 4mq+2p$. Then there exists $Y\in \f$ such that 
\begin{enumerate}
\item
Either $Y$ is $(\epsilon_{2m+1-d},a)$-flexible, $l_{\emptyset,Y}(\f)\geq l(\f)/2 - 2mq - p$ and (if $d\geq 1$) $Y$ is $(\epsilon_{2m+1-d},\s_i)$-fat for all $i$, $0\leq i\leq d-1$, or
\item
$Y$ is $(\epsilon_{2m+1-d},b)$-anti-flexible, $l_{Y,[n]}(\f)\geq l(\f)/2 - 2mq-p$ and (if $d\geq 1$) $[n] \backslash Y$ is $(\epsilon_{2m+1-d},\s_i)$-fat for all $i$, $0\leq i\leq d-1$. 
\end{enumerate}
\end{lemma}
\begin{proof}
Let $\f_-=\{F\in \f: |F|\leq n/2\}$ and $\f_+ = \{F\in \f: |F|\geq n/2 \}$. Clearly either $l(\f_-)\geq l(\f)/2$ or $l(\f_+)\geq l(\f)/2$.

Suppose first that $l(\f_-)\geq l(\f)/2$. Let $\f_-^*$ be the elements of $\f_-$ which are not $(\epsilon_{2m+1-d},a)$-flexible. By Lemma \ref{pivotsbound}, we have $l(\f_-^*)\leq f(\epsilon_{2m+1-d},a)\leq p$ (here notice that if $F$ is not $(\epsilon_{2m+1-d},a)$-flexible in $\f_-$ then it isn't $(\epsilon_{2m+1-d},a)$-flexible in $\f_-^*$ either). If $d\geq 1$, for each $i\in [0,d-1]$ let $\f_-^i$ be the elements of $\f_-$ which are not $(\s_i,\epsilon_{2m+1-d})$-fat; by definition of $\epsilon_{2m+2-d}$, we have $\epsilon_{2m+2-d}\leq \eta(\epsilon_{2m+1-d},r_i)$, and by definition of $q$, we have $q\geq q_{2m+2-d}\geq h(\epsilon_{2m+1-d},r_i)$ so that, by Lemma \ref{nodense}, $l(\f_-^i)\leq q$. Now if $d=0$ let $\f_-^\star = \f_- \backslash \f_-^*$, and if $d\geq 1$  let $\f_-^\star = \f_- \backslash \left( \f_-^*\cup \f_-^0 \cup \f_-^1 \cup \f_-^2 \cup \cdots \cup \f_-^{d-1} \right)$. Either way, we clearly have
$\l(\f_-^\star)\geq l(\f_-) - 2mq - p\geq l(\f)/2 -2mq -p>0$. By Lemma \ref{centred} there exists $Y\in \f_-^\star$ with $\l_{\emptyset,Y}(\f_-^\star)\geq l(\f_-^\star)\geq l(\f)/2-2mq -p$. $Y$ satisifies the requirements of the lemma in this case.

Suppose now that $l(\f_+)\geq l(\f)/2$. Let $\widetilde{\f}=\{[n] \backslash F:F\in \f_+ \}$. Then $|F|\leq n/2$ for every $F\in \widetilde{\f}$ and $l(\widetilde{\f}) = l(\f_+)$. Therefore by the same argument as in the previous case, there exists $Y\in \widetilde{\f}$ which is $(\epsilon_{2m+1-d},\s_i)$-fat for all $i$, $0\leq i\leq d-1$ (if $d\geq 1$), is $(b,\epsilon_{2m+1-d})$-flexible and $l_{\emptyset,Y}(\f)\geq l(\widetilde{\f}) - 2mq -p$. But notice that a $b$-pivot for $Y$ corresponds to a $b$-anti-pivot of $[n] \backslash Y \in \f_+$, so that if we let  $Y' = [n] \backslash Y$ then $Y'$ is $(\epsilon_{2m+1-d},b)$-anti-flexible. Moreover $l_{\emptyset,Y}(\widetilde{\f}) = l_{[n]\backslash Y,[n]}(\f_+)$. Thus $Y'$ satisfies the requirements of the lemma in this case.  
\end{proof}

Let us remark that, while we stated Lemma \ref{intermediate} for a family of subsets of $[n]$ for ease of notation, the lemma obviously remains true for a family of subsets of any set $Z$; in fact, in the next lemma, we shall apply Lemma \ref{intermediate} to families whose ground sets are subsets of $[n]$.  
\begin{lemma}
\label{sequences}
Let $\f\subseteq \p[n]$ with 
$$l(\f)\geq  2^{2m+1}(2m+1) + \sum_{i=1}^{2m+1} 2^{i}(4mq+2p).$$ 
Then there exists $t\in [0,2m]$ and sequences $(\f_i)^t_{i=-1}$, $(A_i)^t_{i=-1}$, $(B_i)^t_{i=-1}$, $(a_i)^t_{i=-1}$, $(b_i)^t_{i=-1}$, $(\s_i)^t_{i=0}$ where $\f_{-1} = \f$, $A_{-1} = [n]$, $B_{-1} = \emptyset$, $a_{-1} = b_{-1} = -1$, and 
\begin{enumerate}
\item
For each $i\in [-1,t]$, $\f_i$ is a subfamily of $\f$ with $B_i\subseteq F\subseteq A_i$ for all $F\in \f$; 
\item
For each $i\in [0,t]$, either $a_i = a_{i-1} + 1$ and $b_i = b_{i-1}$, or $a_i = a_{i-1}$ and $b_i = b_{i-1} + 1$; 
\item
For each $i\in [0,t]$, if $a_i = a_{i-1} + 1$ then $A_i\in \f_{i-1}$, $B_i = B_{i-1}$ and $A_i \backslash B_{i-1}$ is an $(\epsilon_{2m+1-t},a_i)$-flexible set in $(\f_{i-1})_{B_{i-1},A_{i-1}}$ with set of $a_i$-pivots $\s_i$. If on the other hand $b_i = b_{i+1}+1$ then $B_i\in \f_{i-1}$, $A_i = A_{i-1}$ and $B_i \backslash B_{i-1}$ is an $(\epsilon_{2m+1-t},b_i)$-anti-flexible set in $(\f_{i-1})_{B_{i-1},A_{i-1}}$ with set of $b_i$-anti-pivots $\s_i$; 
\item
$A_t \backslash B_t$ is $(\epsilon_{2m+1-t},\s_i)$-fat for each $i \in [0,t]$;
\item
$l_{B_t,A_t}(\f_t)\geq 2^{2m-t}(2m+1)+\sum_{i=1}^{2m-t} 2^{i}(2mq+p)$;
\item
Either $a_t = m$ or $b_t = m$.
\end{enumerate}
\end{lemma}

\begin{proof}

We shall prove that if for some $d\in [0,2m]$, there exist sequences $(\f_i)^{d-1}_{i=-1}$, $(A_i)^{d-1}_{i=-1}$, $(B_i)^{d-1}_{i=-1}$, $(a_i)^{d-1}_{i=-1}$, $(b_i)^{d-1}_{i=-1}$, $(\s_i)^{d-1}_{i=0}$ as above satisfying conditions 1 to 5 of the lemma except condition 6, then we can find $\f_d$, $A_d$, $B_d$, $a_d$, $b_d$ and $\s_d$ to enlarge each sequence, so that the new sequences satisfy conditions 1 to 5 of the lemma and either $a_{d} = a_{d-1} + 1$ or $b_{d} = b_{d-1} + 1$. This is enough to prove the lemma. Indeed, we start with $a_{-1} = -1$ and $b_{-1} = -1$ and whenever we enlarge the sequences, either $a_{d}$ increases by one or $b_{d}$ increases by one. Hence after at most $2m+1$ enlargements, it must be the case that condition 6 of the lemma is satisfied, and we let $t$ be the last value of $d$ obtained. 

So suppose that for some $d\in [0,2m]$ we have sequences $(\f_i)^{d-1}_{i=-1}$, $(A_i)^{d-1}_{i=-1}$, $(B_i)^{d-1}_{i=-1}$, $(a_i)^{d-1}_{i=-1}$, $(b_i)^{d-1}_{i=-1}$, $(\s_i)^{d-1}_{i=0}$ satisfying conditions 1 to 5 of the lemma, but where $a_{d-1}<m$ and $b_{d-1} < m$. By condition 5, $l_{A_{d-1},B_{d-1}}(\f_{d-1})\geq 2(2m+1) + 2(2mq + p) > 4mq + 2p$. Therefore we may apply  Lemma \ref{intermediate} to $(\f_{d-1})_{B_{d-1},A_{d-1}}$, viewed as a family of subsets of $A_{d-1} \backslash B_{d-1}$. As we view the ground set of $(\f_{d-1})_{B_{d-1},A_{d-1}}$ to be $A_{d-1} \backslash B_{d-1}$, when applying the lemma, we take the sequence of $\s_i$'s to be $(\s_i')_{i=0}^{d-1}$, where $\s_i' = \s_i\cap \p(A_{d-1}\backslash B_{d-1})$ for each $i$. We also let $a = a_{d-1} + 1$ and $b = b_{d-1}+1$. There are two possible outcomes from applying the lemma. 
\begin{description}[leftmargin=*]
\item[Case 1.] We obtain $Y \in (\f_{d-1})_{B_{d-1},A_{d-1}}$ which is $(\epsilon_{2m+1-d},\s_i')$-fat for all $i$, $0\leq i\leq d-1$, $Y$ is $(\epsilon_{2m+1-d},a)$-flexible and $l_{\emptyset,Y}((\f_{d-1})_{B_{d-1},A_{d-1}})\geq l((\f_{d-1})_{B_{d-1},A_{d-1}})/2 - 2mq - p$. We let $b_{d} = b_{d-1}$, $a_{d} = a$, $A_{d} = B_{d-1} \cup Y$, $B_{d} = B_{d-1}$, 
$$\f_{d} = \{ F \in \f_{d-1}: B_d \subseteq F\subseteq A_d\},$$ 
and we let $\s_{d}$ be the set of $a$-pivots for $Y$ in $(\f_{d-1})_{B_{d-1},A_{d-1}}$. 

Let us check that this is a valid enlargement of the sequences, i.e. that conditions 1 to 5 of the lemma are satisfied. Conditions 1 and 2 need only be checked for $i=d$ since for smaller $i$ they are inherited from the sequences that we are enlarging. Condition 1 is immediately satisfied by the definition of $\f_d$. Condition 2 is equally trivally satisfied since $a_d = a = a_{d-1} + 1$ and $b_d = b_{d-1}$. Condition 3 also only needs to be checked for $i=d$ because we know it holds for the sequences we are enlarging, and $\epsilon_{2m+1-(d-1)}\leq \epsilon_{2m+1-d}$. Now $Y \in (\f_{d-1})_{B_{d-1},A_{d-1}}$, so that $A_d  = B_{d-1} \cup Y \in \f_{d-1}$, and $A_d\backslash B_{d-1} = Y$, which is $(\epsilon_{2m+1-d},a)$-flexible in $(\f_{d-1})_{B_{d-1},A_{d-1}}$, with sets of $a$-pivots $\s_d$. So condition 3 is satisfied.  Condition 4 is satisfied because $A_d\backslash B_d = Y$, which is $(\epsilon_{2m+1-d},\s_i')$-fat for all $i\in [0,d-1]$ (hence $(\epsilon_{2m+1-d},\s_i)$-fat too), and also $(\epsilon_{2m+1-d},\s_d)$-fat by definition of $\s_d$. Finally, 
\begin{align*}
l_{B_d,A_d}(\f_d) &= l_{\emptyset,A_d\backslash B_d}((\f_d)_{B_d,A_d}) \\
&= l_{\emptyset,Y}((\f_{d-1})_{B_{d-1},A_{d-1}}) \\
&\geq l_{B_{d-1},A_{d-1}}(\f_{d-1})/2 - 2mq - p \\
&\geq  2^{2m-d}(2m+1)+\sum_{i=1}^{2m-d} 2^{i}(2mq+p),
\end{align*}
and so condition 5 is satisfied.
\item[Case 2.] We obtain $Y \in (\f_{d-1})_{B_{d-1},A_{d-1}}$ such that $(A_{d-1} \backslash B_{d-1}) \backslash Y$ is \newline $(\epsilon_{2m+1-d},\s_i')$-fat for all $i$, $0\leq i\leq d-1$, $Y$ is $(\epsilon_{2m+1-d},b)$-anti-flexible, and \newline $l_{Y,B_{d-1} \backslash A_{d-1}}((\f_{d-1})_{B_{d-1},A_{d-1}})\geq l((\f_{d-1})_{B_{d-1},A_{d-1}})/2 - 2mq-p$. Then we let $a_{d} = a_{d-1}$, $b_{d} = b$, $A_{d} = A_{d-1}$, $B_{d} = B_{d-1} \cup Y$,
$$\f_{d}=\{F \in \f_{d-1}: B_d \subseteq F\subseteq A_d\},$$ 
and we let $\s_{d}$ be the set of $b$-anti-pivots for $Y$ in $(\f_{d-1})_{B_{d-1},A_{d-1}}$. 

It is again straightforward, but tedious, to check that this is a valid enlargement of the sequences. \qedhere
\end{description}
\end{proof}

We are now able to finish the proof of Theorem \ref{maintheorem}. Let $n\in \N$ and let $\f\subseteq \p[n]$ with 
$$l(\f)\geq  2^{2m+1}(2m+1) + \sum_{i=1}^{2m+1} 2^{i}(2mq+p).$$
Let $(\f_i)^t_{i=-1}$, $(A_i)^t_{i=-1}$, $(B_i)^t_{i=-1}$, $(a_i)^t_{i=-1}$, $(b_i)^t_{i=-1}$, $(\s_i)^t_{i=0}$ be the sequences satisfying the conclusion of Lemma \ref{sequences}. Notice that it follows from the definition of the sequences that $\emptyset = B_{-1}\subseteq B_0\subseteq B_1\subseteq \cdots \subseteq B_t\subseteq A_t \subseteq A_{t-1}\subseteq \cdots \subseteq A_0\subseteq A_{-1} = [n]$. 

There are two cases to consider: either $a_t = m$ or $b_t = m$. Suppose first that $a_t = m$. Then there exist $i_0 < i_1 < i_2 < \ldots <i_m$ such that for each $k$, $a_{i_k} = k$ and $a_{i_k}> a_{i_k-1}$, so that $\s_{i_k}$ is a set of $k$-pivots for $A_{i_k} \backslash B_{i_k-1}$ in $(\f_{i_k-1})_{B_{i_k-1},A_{i_k-1}}$. Notice that since all the elements of $(\f_{i_k-1})_{B_{i_k-1},A_{i_k-1}}$ are obtained from $\f_{i_k-1}$ by removing $B_{i_k-1}$ from them and also $A_{i_k} \in \f_{i_k-1} $, $\s_{i_k}$ is also a set of $k$-pivots for $A_{i_k}$ in $\f_{i_k - 1}$. 

Now, let $X = A_{i_m} \backslash B_{i_m}$. Since $l_{B_{i_m},A_{i_m}}(\f_{i_m})\geq 2m+1$, $|X|\geq 2m$. By condition 5, $X$ is $(\epsilon_{2m+1-i_m},\s_{i_k})$-fat for all $k$, $0\leq k\leq m$. As $\epsilon_{2m+1}\leq \epsilon_{2m}\leq \cdots\leq \epsilon_1 = \epsilon$, this implies that $X$ is $(\epsilon,\s_{i_k})$-fat for each $k$, $0\leq k\leq m$. For a fixed $k$, let $V_k = \s_{i_k}\cap X^{(k)}$, i.e. $V_k$ is the set of $k$-pivots for $A_{i_k}$ in $\f_{i_k - 1}$ which are contained in $X$. For each $x\in V_k$, let $w_x$ be a witness of $x$ being a $k$-pivot of $A_{i_k}$ in $\f_{i_k-1}$. Let $W_k = \bigcup_{x\in V_k}\{w_x\}$ and let $W = \bigcup_{k=0}^m W_k$. We shall prove the following claim.

\begin{claim}
\label{finalstep}
$W$, viewed as a subposet of $\f$, is isomorphic to an element of $\D(|X|,m,\epsilon)$. 
\end{claim}
\begin{proof}

Order the elements of $V$ so that $x\leq y$ if and only if $x\supseteq y$. Up to relabelling of the elements of $X$ it is clear that $V\in \D(|X|, m, \epsilon)$, since $X$ is $(\epsilon,V_k)$-fat for every $k$, $0\leq k\leq m$.

Let $\psi: V\longrightarrow W$ be the map sending $x\in V$ to $w_x\in W$. We shall prove that $\psi$ is an isomorphism between $(V,\leq)$ and $(W,\subseteq)$, which proves the claim.

Suppose $x,y\in V$ where $w_x\in \f_{i_{k_1}-1}$ for some $i_{k_1}$ and $w_y\in \f_{i_{k_2}-1}$ for some $i_{k_2}$ and without loss of generality $k_1\leq k_2$. If $k_1=k_2$ and $x\neq y$ then $x$ and $y$ are incomparable in $V$ and $w_x$ and $w_y$ are two distinct elements of $\f$ of the same size, hence are also incomparable in $\f$. So let us assume that $k_1<k_2$ and $x\neq y$. Clearly since $|x|<|y|$, either $x>y$ (if $x\subsetneq y$) or $x$ and $y$ are incomparable. If $x\subsetneq y$ then by Obervation \ref{obs} $w_y\subsetneq w_x$ since $w_y\subseteq A_{i_{k_1}}$. If $x$ and $y$ are incomparable in $V$ then there exists $z\in x\backslash y$, but since $x\subseteq X\subseteq A_{i_{k_2}}$ it must be the case that $z\in w_y$, and so since $z\not\in w_x$, $w_x$ and $w_y$ must be incomparable. This shows that $\psi$ preserves the order relation and finishes the proof of the claim. 
\end{proof}
By Claim \ref{finalstep} $\f$ contains an induced poset isomorphic to an element of $\D(|X|,m,\epsilon)$, and since $|X|\geq 2m$ it contains an induced copy of $P$ by Lemma \ref{universality2} and Lemma \ref{universality}. This finishes the proof of Theorem \ref{maintheorem} in the case where $a_t  = m$. 

Suppose now that $b_t = m$. Then there exist $i_0 < i_1 < \ldots < i_m$ such that for each $k$, $b_{i_k} = k$ and $b_{i_k}>b_{i_k - 1}$, so that $\s_{i_k}$ is a set of $k$-anti-pivots for $B_{i_k}$ in $(\f_{i_k - 1})_{B_{i_k-1},A_{i_k-1}}$. As above, it is easily seen that $\s_{i_k}$ is also a set of $k$-anti-pivots for $B_{i_k}$ in $\f_{i_k - 1}$. 

Let $X = A_{i_m} \backslash B_{i_m}$; $|X|\geq 2m$ as before. For a fixed $k$, $0\leq k\leq m$, let $V'_k = \s_{i_k}\cap X^{(k)}$, so $V'_k$ is the set of $k$-anti-pivots of $B_{i_k}$ which are contained in $X$. For $x\in V'_k$ let $w_x$ be a witness of $x$ being a $k$-anti-pivot of $B_{i_k}$ in $\f_{i_k -1}$. Let $W'_k = \bigcup_{x\in V'_k}\{w_x\}$ and let $W' = \bigcup_{k=0}^mW'_k$. In a similar fashion as above, we have the following claim.

\begin{claim}
\label{finalstep2}
$W'$, viewed as a subposet of $\f$, is isomorphic to an element of $\U(|X|,m,\epsilon)$. 
\end{claim}
\begin{proof}
Order the elements of $V'$ so that $x\leq y$ if and only if $x\subseteq y$. Up to relabelling of the elements of $X$ it is clear that $V'\in \U(|X|, m, \epsilon)$, since $X$ is $(\epsilon,V'_k)$-fat for every $k$, $0\leq k\leq m$.

Let $\psi': V'\longrightarrow W'$ be the map sending $x\in V'$ to $w_x\in W'$. We shall prove that $\psi'$ is an isomorphism between $(V',\leq)$ and $(W',\subseteq)$, which proves the claim.

Suppose $x,y\in V'$ where $w_x\in \f_{i_k-1}$ for some $i_{k_1}$ and $w_y\in \f_{i_{k_2}}$ for some $i_{k_2}$ and without loss of generality $k_1\leq k_2$. If $k_1 = k_2$ and $x\neq y$ then $x$ and $y$ are incomparable in $V'$ and $w_x$ and $w_y$ are two distinct elements of $\f$ of the same size, hence are also incomparable in $\f$. So let us assume that $k_1<k_2$ and $x\neq y$. Clearly since $|x|<|y|$, either $x<y$ (if $x\subsetneq y$) or $x$ and $y$ are incomparable. If $x\subsetneq y$, then $w_x \subsetneq w_y$ by Observation \ref{obs} since $B_{i_{k_1}}\subseteq  w_y$. If $x$ and $y$ are incomparable then there exists $z\in x\backslash y$, but since $z\in X \subseteq [n] \backslash B_{i_{k_2}}$, $z\not\in w_y$, and so $w_x$ and $w_y$ must be incomparable. This shows that $\psi'$ preserves the order relation and finishes the proof of the claim.  
\end{proof}
By Claim \ref{finalstep2} $\f$ contains an induced poset isomorphic to an element of $\U(|X|,m,\epsilon)$, and since $|X|\geq 2m$ it contains an induced copy of $P$ by Lemma \ref{universality2} and Lemma \ref{universality}. This finishes the proof of Theorem \ref{maintheorem} in the case where $b_t  = m$. 
\qed
\section{Acknowledgement}
The author wishes to thank Andrew Thomason for his very hepful comments and advice. This research was funded by an EPSRC doctoral studentship.  

\end{document}